\documentclass[10pt,a4paper,reqno]{amsart}
\usepackage[latin1]{inputenc}
\usepackage[T1]{fontenc}
\usepackage{lmodern}
 \usepackage[english]{babel,minitoc}
\usepackage{hyperref}
\usepackage{amsfonts}
\usepackage{amsmath}

\numberwithin{equation}{section}\theoremstyle{plain}
\newtheorem{theorem}{Theorem}[section]

\newtheorem{lemma}[theorem]{Lemma}
\newtheorem{corollary}[theorem]{Corollary}
\newtheorem{definition}[theorem]{Definition}
\newtheorem{remark}[theorem]{Remark}

\address{\begin{center}{\small Department of Mathematics and Computer Sciences, Faculty of Sciences-Mekn\`{e}s,\\
Equipe d'Analyse Harmonique et Probabilit\'{e}s, University Moulay Isma\"{\i}l,\\
BP 11201 Zitoune, Meknes, Morocco}
\end{center}}

\begin{document}

\title[The two-sided Gabor Quaternionic Fourier transform and some uncertainty principles]
{The two-sided Gabor Quaternionic Fourier transform and some uncertainty principles}

\author[ S.Fahlaoui and M.El kassimi]{ S.Fahlaoui  \quad and \quad  M.El kassimi}

\address{Sa\"{\i}d Fahlaoui}  \email{s.fahlaoui@fs.umi.ac.ma}

\address{Mohammed El kassimi} \email{m.elkassimi@edu.umi.ac.ma}

\maketitle

\maketitle
\begin{abstract}
In this paper, we define a new transform called the Gabor  quaternionic Fourier transform (GQFT), which generalizes the classical windowed Fourier transform to quaternion valued-signals, we give several important properties such as the Plancherel formula and  inversion formula. Finally, we establish the Heisenberg uncertainty principles for the GQFT.
\end{abstract} {\it keywords:} Quaternion algebra, Quaternionic Fourier transform, Gabor Fourier transform, Heisenberg uncertainty principle.\\
\section{Introduction}
As it is known, the quaternion Fourier transform (QFT) is a very useful mathematical tool. It has been discussed extensively in the literature and has proved  to be powerful and useful in some theories. In \cite{Assefa,Bulow,Ell0} the authors used the (QFT) to extend the colour image analysis. Researchers in \cite{Bayro} applied the QFT to image processing and neural computing techniques. The QFT is a generalisation of the real and complex Fourier transform (FT), but it  is ineffective in representing and computing local information about quaternionic signals. A lot of papers have been devoted to the extension of the theory of the windowed FT to the quaternionic case. Recently B\"{u}low and Sommer \cite{Bulow,Bulow1} extend the WFT to the quaternion algebra. They introduced a special case of the GQFT known as quaternionic Gabor filters. They applied these filters to obtain a local two-dimensional quaternionic phase. In \cite{Bahri} Bahri et al.
studied the right sided windowed quaternion Fourier transform. In \cite{Fu} the authors  studied two-sided windowed (QFT) for the case when the window has a real valued. Moreover, they also pointed out that the extension of the windowed Fourier transforms to the quaternionic case by means of a two-sided QFT is rather complicated in view of the non-commutativity. So for that, this paper attempts to study the two-sided quaternionic Gabor Fourier transform (GQFT) with the window has a quaternionic valued and some important properties are derived. We start by  reminding some results of two-sided quaternionic Fourier transform (QFT), we give some examples, to show the difference between the GQFT and WFT, and we establish important properties of the GQFT like inversion formula, Plancherel formula, using a version of Heisenberg uncertainty principle for two-sided QFT to prove a generalized uncertainty principle for GQFT. 
\subsection{Definition and properties of quaternion $\mathbb{H}$}: \\
The quaternion algebra $\mathbb{H}$ is defined over $\mathbb{R}$ with three imaginary units i, j and k obey the Hamilton's multiplication rules, \begin{equation}\label{equ1} ij=-ji=k, ~~ jk=-kj=i, ~~ki=-ik=j \end{equation}
$$i^2=j^2=k^2=ijk=-1$$
According to \ref{equ1} $\mathbb{H}$ is non-commutative, one cannot directly extend various results on complex numbers to a quaternion. For simplicity, we express a quaternion q as the sum of scalar $q_1$, and a pure 3D quaternion q.
Every quaternion can be written explicitly as:
$$q=q_1+iq_2+jq_3+kq_4\in \mathbb{H},~~ q_1,q_2,q_3,q_4\in\mathbb{R},$$
The conjugate of quaternion $q$ is obtained by changing the sign of the pure part, i.e.
$$\overline{q}=q_1-iq_2-jq_3-kq_4$$
The quaternion conjugation is a linear anti-involution $$\overline{\overline{p}}=p,~~~~ \overline{p+q}=\overline{p}+\overline{q},~~ \overline{pq}=\overline{q} \overline{p},~ \forall p,q\in \mathbb{H}$$
The modulus $|q|$ of a quaternion q is defined as $$|q|=\sqrt{q\overline{q}}=\sqrt{q_1^2+q_2^2+q_3^2+q_4^2},~~ |pq|=|p||q|. $$
It is straight forward to see that
$$|pq|=|p||q|, |q|=|\overline{q}|, p, q \in \mathbb{H}$$
In particular, when $q=q_1$ is a real number, the module $|q|$ reduces to the ordinary Euclidean modulus, i.e. $|q|=\sqrt{q_1q_1}$.  A function $f:\mathbb{R}^{2}\rightarrow\mathbb{H}$ can be expressed as
 $$f (x, y): =f_1 (x, y) +if_ {2} (x, y) +jf_3 (x, y)+kf_4 (x, y), $$
 where $(x,y)\in\mathbb{R}\times\mathbb{R}$.\\
We introduce an inner product of functions $f, g$ defined on $\mathbb{R}^2$ with values in $\mathbb{H}$ as follows $$<f,g>_{L^{2}(\mathbb{R}^2,\mathbb{H})}=\int_{\mathbb{R}^2}f(x)\overline{g(x)}dx$$
If $f=g$ we obtain the associated norm by $$\|f\|^{2}_{2}=<f,f>_{2}=\int_{\mathbb{R}^2}|f(x)|^2dx$$
The space $L^{2}(\mathbb{R}^2,\mathbb{H})$ is then defined as
$$L^2(\mathbb{R}^2,\mathbb{H})=\{f| f:\mathbb{R}^2\rightarrow\mathbb{H}, \|f\|_{2}<\infty\}$$
And we define the norm of $L^{2}(\mathbb{R}^2,\mathbb{H})$ by $$\|f\|^{2}_{L^{2}(\mathbb{R}^2,\mathbb{H})}=\|f\|^{2}_{2}$$
\section{The two-sided  Gabor Quaternionic  Fourier transform (GQFT)}

The quaternion Fourier transform (QFT) is an extension of Fourier transform proposed by Ell \cite{Ell}. Due to the non-commutative properties of quaternion, there are three different types of QFT, the left sided QFT, the right sided QFT and the two-sided QFT \cite{Ding}. In this paper  we only treat the two-sided QFT. We now review the definition and some properties of the two-sided QFT\cite{Li-ping}.\\

\begin{definition}[Quaternion Fourier transform]
  The two-sided quaternion Fourier transform (QFT)of a quaternion function
  $f\in L^{1}(\mathbb{R}^{2},\mathbb{H})$ is the function  $\mathcal{F}_{q}(f): \mathbb{R}^{2}\rightarrow \mathbb{H}$  defined by:\\ for $\omega=(\omega_1,\omega_2)\in \mathbb{R}\times\mathbb{R}$
  \begin{equation}\label{QFT}
  \mathcal{F}_{q}(f)(w)=\int_{\mathbb{R}^{2}}e^{-2\pi ix_1.\omega_1}f(x)e^{-2\pi jx_2.\omega_2}dx
  \end{equation}
  where $dx=dx_1dx_2$
\end{definition}
This transform can be inverted by means of:
\begin{theorem}\label{QFT-inversion}
If $f, \mathcal{F}_{q}(f)\in L^{2}(\mathbb{R}^2,\mathbb{H})$, then,
 \begin{equation}\label{inversion-QFT}
   f(x)=\mathcal{F}_{q}^{-1}{\mathcal{F}_q(f)}(x)=\int_{\mathbb{R}^{2}}e^{2\pi ix_1.\omega_1}\mathcal{F}_q(f)(\omega)e^{2\pi jx_2.\omega_2}d\omega
 \end{equation}
\end{theorem}

\begin{theorem}[Plancherel theorem for QFT ]\label{placherel-QFT}
If $f\in L^{2}(\mathbb{R}^2,\mathbb{H})$ then
\begin{equation}\label{Planch-QFT}
\|f\|_{2}=\|\mathcal{F}_{q}(f)\|_{2}
\end{equation}
\end{theorem}
\begin{proof}
  See \cite{Li-ping}
\end{proof}
 \begin{definition}
   A quaternion window function is a non null function $\varphi\in L^{2}(\mathbb{R}^{2},\mathbb{H})$  \\
 \end{definition}

Based on the above formula \ref{QFT} for the QFT, we establish the following definition of the two-sided Gabor quaternionic  Fourier transform (GQFT).
\begin{definition}\label{def-GQFT}
  We define the GQFT of $f\in L^{2}(\mathbb{R}^2,\mathbb{H})$ with respect to non-zero quaternion window function
  $\varphi\in L^{2}(\mathbb{R}^2,\mathbb{H})$ as,
  \begin{equation} \label{eqGQFT}
 \mathcal{G}_{\varphi}f(\omega,b)=\int_{\mathbb{R}^2}e^{-2\pi ix_1\omega_1}f(x)\overline{\varphi(x-b)}e^{-2\pi jx_2.\omega_2}dx
  \end{equation}
\end{definition}
Note that the order of the exponentials in \ref{eqGQFT} is fixed because of the non-commutativity of the product of quaternion.\\
The energy density is defined as the modulus square of GQFT \ref{def-GQFT} given by
\begin{equation}\label{energy}
  |G_{\varphi} f(\omega,b)|^2=|\int_{\mathbb{R}^2}e^{-2i\pi x_1\omega_1}f(x)\overline{\varphi(x-b)}e^{-2j\pi x_2\omega_2}dx|^2
\end{equation}

The equation \ref{energy} is often called a spectogram which measures the energy of a quaternion-valued function f in the position-frequency neighbourhood of $(\omega,b)$.

\subsection{Examples of the GQFT}\leavevmode\par
For illustrative purposes, we shall discuss examples of the GQFT. We begin with a straightforward example.\\
\textbf{Example 1}\\
Consider the two-dimensional window function defined by
\begin{equation}
  \varphi(x)=\left\{
  \begin{array}{ll}
    1, ~~\hbox{for}~~ -1\leq x_1\leq 1~~ \hbox{and}~~ -1\leq x_2\leq 1; \\\\

    0, ~~\hbox{otherwise}
  \end{array}
\right.
\end{equation}
Obtain the GQFT of the function defined as follows
\begin{equation}
  f(x)=\left\{
  \begin{array}{ll}
    e^{-x_1-x_2},  0\leq x_1\leq +\infty~~ \hbox{and}~~ 0\leq x_2\leq +\infty; \\\\

    0,~~ \hbox{otherwise}
  \end{array}
\right.
\end{equation}
By applying the definition of the GQFT we have
\begin{eqnarray*}
  G_{\varphi}f(\omega,b) &=& \int_{m_1}^{1+b_1}\int_{m_2}^{1+b_2}e^{-i2\pi x_1\omega_1}e^{-x_1-x_2}e^{-j2\pi x_2 \omega_2}dx_1dx_2, \\
 &~~& \hbox{with}~~ m_1=max(0,-1+b_1); m_2=max(0,-1+b_2), \\
   &=&\int_{m_1}^{1+b_1}e^{-x_1(1+i2\pi\omega_1)}dx_1\int_{m_2}^{1+b_2}e^{-x_2(1+j2\pi\omega_2)}dx_2, \\
   &=& [\frac{e^{-x_1(1+i2\pi\omega_1)}}{(-1-i2\pi\omega_1)}]_{m_1}^{1+b_1}[\frac{e^{-x_2(1+j2\pi\omega_2)}}{(-1-j2\pi\omega_2)}]_{m_2}^{1+b_2}, \\
   &=& \frac{1}{(-1-i2\pi\omega_1)(-1-j2\pi\omega_2)}(e^{-(1+b_1)(1+i2\pi\omega_1)}-e^{-m_1(1+i2\pi\omega_1)})(e^{-(1+b_2)(1+j2\pi\omega_2)}-e^{-m_2(1+j2\pi\omega_2)}).
\end{eqnarray*}

\textbf{Example 2}\\
Given the window function of the two-dimensional Haar function defined by:

\begin{equation}
  \varphi(x)=\left\{
  \begin{array}{ll}
    1, ~~\hbox{for} ~~0\leq x_1\leq \frac{1}{2}~~ \hbox{and}~~ 0\leq x_2\leq \frac{1}{2}; \\\\

    -1,~~\hbox{for} ~~\frac{1}{2}\leq x_1\leq 1~~ \hbox{and}~~ \frac{1}{2}\leq x_2\leq 1; \\\\

    0, ~~\hbox{otherwise};
  \end{array}
\right.
\end{equation}

find the GQFT of the Gaussian function $f(x)=e^{-(x_1^2+x_2^2)}$.\\
From definition \ref{def-GQFT} we obtain
\begin{eqnarray*}
 \mathcal{G}_{\varphi}\{f\}(\omega,b)&=&\int_{\mathbb{R}^2} e^{-i2\pi x_1\omega_1}f(x)\overline{\varphi(x-b)} e^{-j2\pi x_2\omega_2} dx, \\
    &=& \int_{b_1}^{\frac{1}{2}+b_1} e^{-i2\pi x_1\omega_1}e^{-x_1}dx_1\int_{b_2}^{\frac{1}{2}+b_2}e^{-x_2} e^{-j2\pi x_2\omega_2}dx_2, \\
    &~~&- \int_{\frac{1}{2}+b_1}^{1+b_1} e^{-i2\pi x_1\omega_1}e^{-x_1}dx_1\int_{\frac{1}{2}+b_2}^{1+b_2}e^{-x_2} e^{-j2\pi x_2\omega_2}dx_2,
\end{eqnarray*}
by completing squares, we have
\begin{eqnarray*}
\mathcal{G}_{\varphi}\{f\}(\omega,b) &=& \int_{b_1}^{\frac{1}{2}+b_1} e^{-(x_1+i\pi \omega_1)^2}e^{-(\omega_1\pi)^2}dx_1\int_{b_2}^{\frac{1}{2}+b_2} e^{-(x_2+j\pi \omega_2)^2}e^{-(\omega_2\pi)^2}dx_2 \\
 &~~&-  \int_{\frac{1}{2}+b_1}^{1+b_1} e^{-(x_1+i\pi \omega_1)^2}e^{-(\omega_1\pi)^2}dx_1\int_{\frac{1}{2}+b_2}^{1+b_2}e^{-(x_2+j\pi \omega_2)^2}e^{-(\omega_2\pi)^2}dx_2
\end{eqnarray*}
making the substitutions $y_1=x_1+i\pi\omega_1$ and $y_2=x_2+j\pi\omega_2$ in the above expression we immediately obtain :
\begin{align}
  &\mathcal{G}_{\varphi}\{f\}(\omega,b) = e^{-(\omega_1^2+\omega_2^2)\pi^2}(\int_{b_1+i\pi\omega_1}^{\frac{1}{2}+b_1+i\pi\omega_1} e^{-y_1^2}dy_1\int_{b_2+j\pi\omega_2}^{\frac{1}{2}+b_2+j\pi\omega_2} e^{-y_2^2}dy_2
 -  \int_{\frac{1}{2}+b_1+i\pi\omega_1}^{1+b_1+i\pi\omega_1} e^{-y_1^2}dy_1 \int_{\frac{1}{2}+b_2+j\pi\omega_2}^{1+b_2+j\pi\omega_2}e^{-y_2^2}dy_2, ) \nonumber\\
  &= e^{-(\omega_1^2+\omega_2^2)\pi^2}( \int_{0}^{1+b_1+i\pi\omega_1}(- e^{-y_1^2})dy_1 +\int_{0}^{\frac{1}{2}+b_1+i\pi\omega_1} e^{-y_1^2}dy_1)
  \times  (\int_{0}^{1+b_2+j\pi\omega_2}(-e^{-y_2^2})dy_2+\int_{0}^{\frac{1}{2}+b_2+j\pi\omega_2} e^{-y_2^2}dy_2) \nonumber\\
   &-  e^{-(\omega_1^2+\omega_2^2)\pi^2}( \int_{0}^{\frac{1}{2}+b_1+i\pi\omega_1}(- e^{-y_1^2})dy_1 +\int_{0}^{1+b_1+i\pi\omega_1} e^{-y_1^2}dy_1)
\times (\int_{0}^{\frac{1}{2}+b_2+j\pi\omega_2}(-e^{-y_2^2})dy_2+\int_{0}^{1+b_2+j\pi\omega_2} e^{-y_2^2}dy_2),\label{exemple2}
\end{align}
Equation \ref{exemple2} can be written in the form

\begin{align*}
\mathcal{G}_{\varphi}\{f\}(\omega,b)=e^{-(\omega_1^2+\omega_2^2)\pi^2}\{[-qf(1+b_1+i\pi\omega_1)+qf(\frac{1}{2}+b_1+i\pi\omega_1)]
\times&[-qf(1+b_2+j\pi\omega_2)+qf(\frac{1}{2}+b_2+j\pi\omega_2)]   \\
  -[-qf(\frac{1}{2}+b_1+i\pi\omega_1)+qf(1+b_1+i\pi\omega_1)]  \times&[-qf(\frac{1}{2}+b_2+j\pi\omega_2)+qf(1+b_2+j\pi\omega_2)]\}
\end{align*}

Where, ~~$qf(x)=\int_{0}^{x}e^{-t^2}dt.$

\section{Properties of GQFT}
In this section, we are going to to give some properties for the Gabor quaternionic  Fourier transform.
\begin{theorem}
  Let $f\in L^{2}(\mathbb{R}^2,\mathbb{H})$; and $\varphi\in L^{2}(\mathbb{R}^2,\mathbb{H})$ be a non zero quaternionic window function. Then, we have
 $$(\mathcal{G}_{\varphi}\{T_y f\}(\omega,b)=e^{-2i\pi y_1\omega_1} (G_\varphi f)(\omega,b-y)e^{-2j\pi x_2.\omega_2}$$
 where $T_yf(x)=f(x-y)$; and $y=(y_1,y_2)\in\mathbb{R}^2$
\end{theorem}
\begin{proof}
  We have $$\mathcal{G}_{\varphi}\{T_y f\}(w,b)=\int_{\mathbb{R}^2} e^{-2\pi ix_1\omega_1}f(x)\overline{\varphi(x-b)}e^{-2\pi j x_2.\omega_2}dx$$
  we take $t=x-y$, then
 \begin{eqnarray*}
   \mathcal{G}_{\varphi}\{T_y f\}(w,b) &=& \int_{\mathbb{R}^2} e^{-2i\pi( t_1+y_1)\omega_1}f(x)\overline{\varphi(t+y-b)}e^{-2j\pi( t_2+y_2)\omega_2}dt \\
    &=& e^{-2i\pi y_1\omega_1}  \int_{\mathbb{R}^2} e^{-2i\pi t_1\omega_1}f(x)\overline{\varphi(t+y-b)}e^{-2j\pi t_2 \omega_2}dt ~~e^{-2i\pi y_2\omega_2}\\
    &=& e^{-2i\pi y_1\omega_1} \mathcal{G}_{\varphi}\{f\}(\omega,b-y)e^{-2j\pi x_2.\omega_2}.
 \end{eqnarray*}
\end{proof}
\begin{theorem}
  Let $\varphi\in L^{2}(\mathbb{R}^2,\mathbb{H})$ a quaternion window function. Then we have
  $$G_{\widetilde{\varphi}}(\widetilde{f})(\omega,b)=\mathcal{G}_{\varphi}\{f\}(-\omega,-b)$$
  Where $\widetilde{\varphi}(x)=\varphi(-x)$; $\forall \varphi\in  L^{2}(\mathbb{R}^2,\mathbb{H})$
\end{theorem}
\begin{proof}
  A direct calculation allows us to obtain for every $f\in  L^{2}(\mathbb{R}^2,\mathbb{H})$
\begin{eqnarray*}
  G_{\widetilde{\varphi}}(\widetilde{f})(\omega,b) &=& \int_{\mathbb{R}^2} e^{-2i\pi x_1\omega_1}f(-x)\overline{\varphi(-(x-b))}e^{-2j\pi x_2 \omega_2}dx \\
   &=&  \int_{\mathbb{R}^2} e^{-2i\pi(- x_1)(-\omega_1)}f(-x)\overline{\varphi(-x-(-b))}e^{-2j\pi (-x_2) (-\omega_2)}dx  \\
   &=& \mathcal{G}_{\varphi}\{f\}(-\omega,-b)
\end{eqnarray*}

\end{proof}
For establishing an inversion formula and Plancherel identity for GQFT we use the fact that, the GQFT can be expressed in terms  of two-sided quaternionic Fourier transform.
$$\mathcal{G}_{\varphi}\{f\}(\omega,b)=\mathcal{F}_{q}\{f(.)\varphi(.-b)\}(\omega)$$

\begin{theorem}[Inversion formula]
  Let $\varphi$ be a quaternion window function. Then for every function $f\in L^{2}(\mathbb{R}^2,\mathbb{H})$ can be reconstructed  by :
  $$f(x)=\frac{1}{\|\varphi\|^2_{2}}\int_{\mathbb{R}^2}\int_{\mathbb{R}^2}e^{2i\pi x_1\omega_1}G_\varphi f(w,b)e^{2j\pi x_2\omega_2}\varphi(x-b)d\omega db$$
\end{theorem}
\begin{proof}
  We have $$\mathcal{G}_{\varphi}\{f\}(\omega,b)=\int_{\mathbb{R}^2}e^{-2i\pi x_1\omega_1} f(x)\overline{\varphi(x-b)}e^{-2j\pi x_2\omega_2}dx$$
  then
  \begin{equation}\label{equa2}
   \mathcal{G}_{\varphi}\{f\}(\omega,b)=\mathcal{F}_{q}(f(x)\overline{\varphi(x-b)})
  \end{equation}
  Taking the inverse of two-sided QFT of both sides of \ref{equa2} we obtain
  \begin{eqnarray}
     f(x)\overline{\varphi(x-b)} &=& \mathcal{F}^{-1}_{q}{G_\varphi f(\omega,b)}(x) \nonumber\\
     &=& \int_{\mathbb{R}^2}e^{2i\pi x_1\omega_1}\mathcal{G}_{\varphi}\{f\}(\omega,b)e^{2j\pi x_2\omega_2}d\omega, \label{equ3}
  \end{eqnarray}

Multiplying both sides of \ref{equ3} from the right and integrating with respect to $db$  we get

$$f(x)\int_{\mathbb{R}^2}|\varphi(x-b)|^2db=\int_{\mathbb{R}^2}\int_{\mathbb{R}^2}e^{2i\pi x_1\omega_1}G_\varphi f(w,b)e^{2j\pi x_2\omega_2}\varphi(x-b)d\omega db$$
then,
$$f(x)=\frac{1}{\|\varphi\|^2_{2}}\int_{\mathbb{R}^2}\int_{\mathbb{R}^2}e^{2i\pi x_1\omega_1}G_\varphi f(w,b)e^{2j\pi x_2\omega_2}\varphi(x-b)d\omega db$$
Set $C_\varphi=\|\varphi\|^2_{\mathbb{R}^2}$ and assume that $0<C_\varphi<\infty$. Then the inversion formula can also written as

$$f(x)=\frac{1}{C_\varphi}\int_{\mathbb{R}^2}\int_{\mathbb{R}^2}e^{2i\pi x_1\omega_1}G_\varphi f(w,b)e^{2j\pi x_2\omega_2}\varphi(x-b)d\omega db$$

\end{proof}
\begin{theorem}[Plancherel theorem ]\label{Parseval-GQFT}
  Let $\varphi$ be quaternion window function and\\ $f\in L^{2}(\mathbb{R}^2,\mathbb{H})$, then we have
  \begin{equation}\label{parseval1}
  \|\mathcal{G}_{\varphi}\{f\}\|^2_{2}=\|f\|^2_{2}\|\varphi\|^2_{2}
  \end{equation}
\end{theorem}
\begin{proof}
  We have
  \begin{eqnarray}
    \|\mathcal{G}_{\varphi}\{f\}\|^{2}_{2} &=& \|\mathcal{F}_q(f(x)\overline{\varphi(x-b)})\|^{2}_{2} \nonumber \\
     &=& \|f(x)\overline{\varphi(x-b)}\|^{2}_{2} \label{pass1}\\
     &=&  \int_{\mathbb{R}^2}\int_{\mathbb{R}^2}|f(x)|^2|\varphi(x-b)|^2dxdb \nonumber\\
     &=& \int_{\mathbb{R}^2} |f(x)|^2 dx\int_{\mathbb{R}^2}|\varphi(t)|^2dt \label{pass2}\\
     &=& \|f\|^2_{2}\|\varphi\|^2_{2}\nonumber
  \end{eqnarray}
where in line \ref{pass1} we use the the Plancherel theorem's of QFT \ref{placherel-QFT}.
\end{proof}
\section{Uncertainty Principles For the GQFT}
In this section we demonstrate some versions of uncertainty principles and inequalities for the two sided quaternion windowed Fourier transform.
\subsection{Heisenberg Uncertainty principle}\leavevmode\par
Before proving the Heisenberg uncertainty principle for GQFT, first, we are giving a version of Heisenberg uncertainty for the QFT, that we will use it to demonstrate our result.
 \begin{theorem}\label{Heis-QFT}
   Let $f\in L^{2}(\mathbb{R}^2,\mathbb{H})$ be a quaternion-valued signal such that :\\
   $ x_k f, \frac{\partial}{\partial x_k}f\in L^{2}(\mathbb{R}^2,\mathbb{H})$ for $k=1,2$, then,
   \begin{equation}\label{Heisenberg-QFT}
   \left(\int_{\mathbb{R}^2} x^2_{k}|f(x)|^2 dx\right)^{\frac{1}{2}}\left(\int_{\mathbb{R}^2}\omega^2_k|\mathcal{F}_q(f)(\omega)|^2d\omega\right)^{\frac{1}{2}}\geq \frac{1}{4\pi}\|f\|^2_{2},
   \end{equation}
 \end{theorem}
 To prove this theorem, we need the following result,
 \begin{lemma}\label{lemma1}
   Let $f\in L^{1}\cap L^{2}(\mathbb{R}^2,\mathbb{H})$. If $\frac{\partial}{\partial x_k}f$ exist and belong to
   $ L^{2}(\mathbb{R}^2,\mathbb{H})$ for $k=1,2$. Then
   \begin{equation}\label{lem-equ}
    (2\pi)^2\int_{\mathbb{R}^2}\omega_k^{2}|\mathcal{F}(f(x))(\omega)|^2 d\omega=\int_{\mathbb{R}^2}|\frac{\partial}{\partial x_k}f(x)|^2dx.
   \end{equation}
 \end{lemma}
\begin{proof}
  See \cite{you}.
\end{proof}
 We are going to prove the first theorem \ref{Heis-QFT}.
 \begin{proof}
   For $k\in{1,2}$. First, by applying lemma \ref{lemma1} and Plancherel's theorem \ref{parseval1}, we obtain

     $$\frac{\int_{\mathbb{R}^2}x_k^2|f(x)|^2dx\int_{\mathbb{R}^2}\omega_k^2|\mathcal{F}_q(f)(\omega)|^2d^2\omega}{\int_{\mathbb{R}^2}|f(x)|^2dx\int_{\mathbb{R}^2}|\mathcal{F}_q(f)(\omega)|^2d\omega} =$$
      \begin{eqnarray*}
      &=&\frac{\frac{1}{(2\pi)^2}\int_{\mathbb{R}^2}x_k^2|f(x)|^2dx\int_{\mathbb{R}^2}|\frac{\partial}{\partial x_k}f(x)|^2d\omega}{\int_{\mathbb{R}^2}|f(x)|^2dx\int_{\mathbb{R}^2}|\mathcal{F}_q(f)(\omega)|^2d\omega}  \\
      &=& \frac{\frac{1}{(2\pi)^2}\int_{\mathbb{R}^2}x_k^2|f(x)|^2dx\int_{\mathbb{R}^2}|\frac{\partial}{\partial x_k}f(x)|^2d\omega}{(\int_{\mathbb{R}^2}|f(x)|^2dx)^2} \\
    &\geq& \frac{1}{16\pi^2}\frac{(\int_{\mathbb{R}^2}(\frac{\partial}{\partial x_k}f(x)x_k\overline{f(x)}+x_kf(x)\frac{\partial}{\partial x_k}\overline{f(x)})dx)^2}{\|f(x)\|^4_{2}} \\
      &=& \frac{1}{16\pi^2}\frac{(\int_{\mathbb{R}^2}x_k\frac{\partial}{\partial x_k}(f(x)\overline{f(x)})dx)^2}{\|f(x)\|^4_{2}}
   \end{eqnarray*}
   Second, using integration par parts, we further get,
   \begin{eqnarray*}
     &=&  \frac{1}{16\pi^2}\frac{([\int_{\mathbb{R}}x_k|f(x)|^2dx_l]_{x_k=-\infty}^{x_k=+\infty}-\int_{\mathbb{R}^2} \|f(x)\|^2dx)^2 }{\|f(x)\|^4_{2}} \\
      &=&  \frac{1}{16\pi^2}
   \end{eqnarray*}
   then,
   \begin{equation*}
   \left(\int_{\mathbb{R}^2} x^2_{k}|f(x)|^2 dx\right)^{\frac{1}{2}}\left(\int_{\mathbb{R}^2}\omega^2_k|\mathcal{F}_q(f)(\omega)|^2d\omega\right)^{\frac{1}{2}}\geq \frac{1}{4\pi}\|f\|^2_{2}
   \end{equation*}
 \end{proof}

  Applying the Plancherel theorem for the QFT \ref{placherel-QFT} to the right-hand side of \ref{Heis-QFT},  we get the following corollary,
  \begin{corollary}\label{coro1}
  Under the above assumptions, we have
  \begin{equation}\label{ineq-QFT}
  \left(\int_{\mathbb{R}^2} x^2_{k}|\mathcal{F}_q^{-1}\{\mathcal{F}_q(f)\}(x)|^2 dx\right)^{\frac{1}{2}}\left(\int_{\mathbb{R}^2}\omega^2_k|\mathcal{F}_q(f)(\omega)|^2d\omega\right)^{\frac{1}{2}}\geq \frac{1}{4\pi}\|\mathcal{F}_q(f)\|^2_{2}
  \end{equation}

\end{corollary}
Now, we are going to  establish a generalization of the Heisenberg type uncertainty principle for the GQFT.
\begin{theorem}[Heisenberg for GQFT]\label{th-Heisenberg}
  Let $\varphi  \in L^{2}(\mathbb{R}^2,\mathbb{H})$ be a quaternion window function and let $\mathcal{G}_{\varphi}\{f\}\in L^{2}(\mathbb{R}^2,\mathbb{H})$ be the GQFT of $f$ such that $\omega_k\mathcal{G}_{\varphi}\{f\}\in L^{2}(\mathbb{R}^2,\mathbb{H}),~~ k=1,2.$ Then for every $f\in L^{2}(\mathbb{R}^2,\mathbb{H})$ we have the following inequality
  \begin{equation}\label{Heis-GQFT}
  \left(\int_{\mathbb{R}^2} x^2_{k}|f(x)|^2 dx\right)^{\frac{1}{2}}\left(\int_{\mathbb{R}^2}\int_{\mathbb{R}^2}\omega^2_k|\mathcal{G}_{\varphi}\{f\}(\omega,b)|^2d\omega db\right)^{\frac{1}{2}}\geq \frac{1}{4\pi}\|f\|^2_{2} \| \varphi\|_{2}
  \end{equation}
\end{theorem}
In order to prove this theorem, we need to introduce the following lemmas. The first lemma called the Cauchy-Schwartz inequality,
\begin{lemma}\label{cauchy}
  Let $f,g\in L^{2}(\mathbb{R}^2,\mathbb{H}) $ be two quaternion on valued functions. Then the Cauchy-Schwartz inequality takes the form
  $$|\int_{\mathbb{R}^2} \overline{f(x)}g(x) dx|^2\leq \int_{\mathbb{R}^2} |f(x)|^2 dx  \int_{\mathbb{R}^2} |g(x)|^2 dx$$
\end{lemma}

\begin{lemma}
  Under the assumptions of theorem \ref{th-Heisenberg}, we have
  \begin{equation}\label{lemme2}
   \|\varphi\|^2_{2}\int_{\mathbb{R}^2} x^2_{k}|f(x)|^2 dx=\int_{\mathbb{R}^2}\int_{\mathbb{R}^2}x_k^2|\mathcal{F}_q^{-1}\{\mathcal{G}_{\varphi}\{f\}(\omega,b)\}(x)|^2dxdb
  \end{equation}
  for $k=1,2$.
\end{lemma}

\begin{proof}
  Applying elementary properties of quaternion, we get
  \begin{eqnarray*}
     \|\varphi\|^2_{2}\int_{\mathbb{R}^2} x^2_{k}|f(x)|^2 dx &=&\int_{\mathbb{R}^2} x^2_{k}|f(x)|^2 dx \int_{\mathbb{R}^2}|\varphi(x-b)|^2db\\
     &=& \int_{\mathbb{R}^2}\int_{\mathbb{R}^2} x^2_{k}|f(x)|^2  |\varphi(x-b)|^2dx db \\
     &=& \int_{\mathbb{R}^2}\int_{\mathbb{R}^2} x^2_{k}|f(x)\overline{\varphi(x-b)}|^2dx db \\
     &=& \int_{\mathbb{R}^2}\int_{\mathbb{R}^2} x^2_{k}|\mathcal{F}^{-1}(\mathcal{G}_{\varphi}\{f\}(\omega,b))(x)|^2dx db
  \end{eqnarray*}
 \end{proof}
  Now, we are going to prove the theorem \ref{Heis-GQFT}.
 \begin{proof}(of theorem \ref{Heis-GQFT})
  Replacing the QFT of $f$ by the GQFT of the left hand side of \ref{ineq-QFT} in corollary \ref{coro1}, we obtain
 \begin{equation}\label{eneq1}
  \left(\int_{\mathbb{R}^2} x^2_{k}|\mathcal{F}_q^{-1}\{\mathcal{G}_{\varphi}\{f\}(\omega,b)\}(x)|^2 dx\right)\left(\int_{\mathbb{R}^2}\omega^2_k|\mathcal{G}_{\varphi}\{f\}(\omega,b)|^2d\omega\right)\geq \frac{1}{16\pi^2}\left(\int_{\mathbb{R}^2}|G_\varphi f(\omega,b)|^2d\omega\right)^2
  \end{equation}
  we have,
 $$ \mathcal{F}^{-1}(\mathcal{G}_{\varphi}\{f\}(\omega,b))(x)=f(x)\overline{\varphi(x-b)}$$
  Taking the square root on both sides of \ref{eneq1} and integrating both sides with respect to $db$ we get
  \begin{equation}\label{eneq2}
   \int_{\mathbb{R}^2}  \left(\int_{\mathbb{R}^2} x^2_{k}|\mathcal{F}_q^{-1}\{\mathcal{G}_{\varphi}\{f\}(\omega,b)\}(x)|^2 dx\right)^{\frac{1}{2}} \left(\int_{\mathbb{R}^2}\omega^2_k|\mathcal{G}_{\varphi}\{f\}(\omega,b)|^2d\omega\right)^{\frac{1}{2}}db\geq \frac{1}{4\pi}\int_{\mathbb{R}^2}\int_{\mathbb{R}^2}|\mathcal{G}_{\varphi}\{f\}(\omega,b)|^2d\omega db
  \end{equation}

  Applying the  Cauchy-Schwartz inequality \ref{cauchy} to the left-hand side of \ref{eneq2} we obtain
  \begin{equation}\label{eneq3}
    \left(\int_{\mathbb{R}^2}\int_{\mathbb{R}^2} x^2_{k}|\mathcal{F}_q^{-1}\{\mathcal{G}_{\varphi}\{f\}(\omega,b)\}(x)|^2 dxdb\right)^{\frac{1}{2}}\left(\int_{\mathbb{R}^2}\int_{\mathbb{R}^2}\omega^2_k\mathcal{G}_{\varphi}\{f\}(\omega,b)|^2d\omega db\right)^{\frac{1}{2}}\geq \frac{1}{4\pi}\int_{\mathbb{R}^2}\int_{\mathbb{R}^2}|\mathcal{G}_{\varphi}\{f\}(\omega,b)|^2d\omega db
  \end{equation}

  Using lemma \ref{lemme2} into the second term on the left-hand side of \ref{eneq3}, and use the Plancherel's formula \ref{Parseval-GQFT} into the right-hand side of \ref{eneq3}, we obtain that

  \begin{equation}\label{eneq5}
    \left( \| \varphi\|^{2}_{2}\int_{\mathbb{R}^2} x^2_{k}| f(x)|^2 dx\right)^{\frac{1}{2}} \left(\int_{\mathbb{R}^2}\int_{\mathbb{R}^2}\omega^2_k|\mathcal{G}_{\varphi}\{f\}(\omega,b)|^2d\omega db\right)^{\frac{1}{2}}\geq \frac{1}{4\pi}\| f\|^2_{2}
    \| \varphi\|^{2}_{2}
  \end{equation}
 Now, simplifying  both sides of \ref{eneq5} by $ \| \varphi\|_{2}$, we get our result.

 \end{proof}

\section{Uncertainty Principles}
\begin{definition}
  A couple $\alpha=(\alpha_1,\alpha_2)$ of non negative integers is called a multi-index. One denotes
  $$|\alpha|=\alpha_1+\alpha_2 ~~ and ~~ \alpha!=\alpha_1!\alpha_2!$$
  and, for $x\in \mathbb{R}^2$
  $$x^{\alpha}=x_1^{\alpha_1}x_2^{\alpha_2}$$
  Derivatives are conveniently expressed by multi-indices
 $$ \partial^{\alpha}=\frac{\partial^{|\alpha|}}{\partial x_1^{\alpha_1}\partial x_2^{\alpha_2}}$$

\end{definition}
Next, we obtain the Schwartz space as (\cite{Kou})
$$\mathcal{S}(\mathbb{R}^2,\mathbb{H})=\{f\in C^{\infty}(\mathbb{R}^2,\mathbb{H}): sup_{x\in\mathbb{R}^2}(1+|x|^k)|\partial^{\alpha}f(x)|<\infty\},$$
where  $C^{\infty}(\mathbb{R}^2,\mathbb{H})$ is the set of smooth function from $\mathbb{R}^2$ to $\mathbb{H}$.

we have the logarithmic uncertainty principle for the QFT \cite{Chen-kou} as follows

\begin{theorem}[QFT logarithmic uncertainty principle ]\label{theo-log1}
 For $f\in \mathcal{S}(\mathbb{R}^2,\mathbb{H})$, we have
 \begin{equation}\label{log3}
   \int_{\mathbb{R}^2}\!ln|x||f(x)|^2dx +\int_{\mathbb{R}^2}\!ln|\omega||\mathcal{F}_Q\{f\}(\omega)|^2d\omega\geq \left(\frac{\Gamma^{'}(t)}{\Gamma(t)}-ln\pi\right)\int_{\mathbb{R}^2}\!|f(x)|^2dx,
  \end{equation}
 Where $\Gamma^{'}(t)=\left( \frac{d}{dt}\right)$ and $\Gamma(t)$ is Gamma function.\\
\end{theorem}

\begin{remark}
  If we apply Plancherl's theorem for QFT \cite{} to the right hand side of \ref{log3}, we get
  \begin{equation}\label{log10}
   \int_{\mathbb{R}^2}\!ln|x||f(x)|^2dx +\int_{\mathbb{R}^2}\!ln|\omega||\mathcal{F}_Q\{f\}(\omega)|^2d\omega\geq \left(\frac{\Gamma^{'}(t)}{\Gamma(t)}-ln\pi\right)\int_{\mathbb{R}^2}\!|\mathcal{F}_Q\{f\}(\omega)|^2d\omega dx,
  \end{equation}

\end{remark}

\begin{lemma}\label{log-lemm1}
  Let $\varphi\in \mathcal{S}(\mathbb{R}^2,\mathbb{H})$ a windowed quaternionic function and $f\in \mathcal{S}(\mathbb{R}^2,\mathbb{H})$. We have
 \begin{equation}\label{log1}
  \int_{\mathbb{R}^2}\int_{\mathbb{R}^2}ln|x||\mathcal{F}^{-1}_Q\{G_{\varphi}f(\omega,b)\}(x)|^2dxdb = \|\varphi\|^2_{L^{2}(\mathbb{R}^2,\mathbb{H})}\int_{\mathbb{R}^2}ln|x||f(x)|^2dx
   \end{equation}
\end{lemma}

\begin{proof}
  By a simple calculation we get,
\begin{eqnarray}
    \int_{\mathbb{R}^2}\int_{\mathbb{R}^2}ln|x||\mathcal{F}^{-1}_Q\{G_{\varphi}f(\omega,b)\}(x)|^2 dx db
    &=& \int_{\mathbb{R}^2}\int_{\mathbb{R}^2}ln|x||f(x)\overline{\varphi(x-b)}|^2 dx db\nonumber  \\
    &=&\int_{\mathbb{R}^2}\int_{\mathbb{R}^2}ln|x||f(x)|^2|\varphi(x-b)|^2 dx db \nonumber  \\
    &=&\int_{\mathbb{R}^2}ln|x||f(x)|^2(\int_{\mathbb{R}^2}|\varphi(x-b)|^2db)dx \label{log2} \\
     &=&\|\varphi\|^2_{L^{2}(\mathbb{R}^2,\mathbb{H})}\int_{\mathbb{R}^2}ln|x||f(x)|^2 dx. \nonumber
\end{eqnarray}

  To obtain the result of lemma \ref{log-lemm1} we use a substitution in \ref{log2}.
\end{proof}

\begin{corollary}
  For $f\in \mathcal{S}(\mathbb{R}^2,\mathbb{H})$, and $\varphi \in \mathcal{S}(\mathbb{R}^2,\mathbb{H})$, we have
   \begin{equation}\label{log4}
   \int_{\mathbb{R}^2}\!ln|x||\mathcal{F}^{-1}_Q\mathcal{F}_Q(f)(x)|^2dx+\int_{\mathbb{R}^2}\!ln|\omega||\mathcal{F}_Q\{f\}(\omega)|^2d\omega\geq \left(\frac{\Gamma^{'}(t)}{\Gamma(t)}-ln\pi\right)\!\int_{\mathbb{R}^2}\!|\mathcal{F}_Q(\omega)|^2d\omega,
  \end{equation}
\end{corollary}

\begin{theorem}
  Let $f\in \mathcal{S}(\mathbb{R}^2,\mathbb{H})$ and $\varphi \in \mathcal{S}(\mathbb{R}^2,\mathbb{H})$ a quaternionc windowed function, we have the following algorithmic inequality,

  \begin{equation}\label{log9}
   \|\varphi\|^2_{L^{2}(\mathbb{R}^2,\mathbb{H})}\! \int_{\mathbb{R}^2}\!ln|x|
    | f(x)|^2dx + \int_{\mathbb{R}^2}\!\int_{\mathbb{R}^2}\!ln|\omega||G_\varphi f(\omega,b)|^2d\omega db  \geq \|\varphi\|^2_{L^{2}(\mathbb{R}^2,\mathbb{H})}\!\left(\frac{\Gamma^{'}(t)}{\Gamma(t)}-ln\pi\right) \int_{\mathbb{R}^2}\int_{\mathbb{R}^2}\!| f(x)|^2dx,
\end{equation}

\end{theorem}

\begin{proof}
  For classical two-sided quaternionic Fourier transform, by theorem \ref{theo-log1},
   \begin{equation}\label{log5}
   \int_{\mathbb{R}^2}ln|x||f(x)|^2dx+\int_{\mathbb{R}^2}ln|\omega||\mathcal{F}_Q\{f\}(\omega)|^2d\omega\geq \left(\frac{\Gamma^{'}(t)}{\Gamma(t)}-ln\pi\right)\int_{\mathbb{R}^2}|f(x)|^2dx,
  \end{equation}

we replace $f$ by $G_\varphi f$ on both sides of \ref{log5}, we get

\begin{equation}\label{log6}
   \int_{\mathbb{R}^2}\!ln|\omega||G_\varphi f(\omega,b)|^2d\omega+\int_{\mathbb{R}^2}\!ln|x||\mathcal{F}_Q\{G_\varphi f\}(x)|^2dx\geq \left(\frac{\Gamma^{'}(t)}{\Gamma(t)}-ln\pi\right)\!\int_{\mathbb{R}^2}|G_\varphi f(\omega,b)|^2dx,
\end{equation}

  Integrating both sides of this equation with respect to $db$, we obtain

 \begin{equation}\label{log7}
    \int_{\mathbb{R}^2}\!\int_{\mathbb{R}^2}\!ln|\omega||G_\varphi f(\omega,b)|^2d\omega db + \int_{\mathbb{R}^2}\!\int_{\mathbb{R}^2}\!ln|x||\mathcal{F}_Q\{G_\varphi f\}(x)|^2dx db \geq \left(\frac{\Gamma^{'}(t)}{\Gamma(t)}-ln\pi\right)\! \int_{\mathbb{R}^2}\!\int_{\mathbb{R}^2}|G_\varphi f(\omega,b)|^2dxdb,
\end{equation}
 Applying lemma \ref{log-lemm1} into the second term on the left hand side of \ref{log7}, yields

   \begin{equation}\label{log8}
    \int_{\mathbb{R}^2}\!\int_{\mathbb{R}^2}\!ln|\omega||G_\varphi f(\omega,b)|^2d\omega db +\|\varphi\|^2_{L^{2}(\mathbb{R}^2,\mathbb{H})} \!\int_{\mathbb{R}^2}\!ln|x|
    | f(x)|^2dx \geq \!\left(\frac{\Gamma^{'}(t)}{\Gamma(t)}-ln\pi\right)\!\|\varphi\|^2_{L^{2}(\mathbb{R}^2,\mathbb{H})} \! \int_{\mathbb{R}^2}\!\int_{\mathbb{R}^2}\!|G_\varphi f(\omega,b)|^2dxdb,
\end{equation}

we applying the Placncherel formula, we obtain our desired result,

 \begin{equation}\label{log9}
   \|\varphi\|^2_{L^{2}(\mathbb{R}^2,\mathbb{H})}\! \int_{\mathbb{R}^2}\!ln|x|
    | f(x)|^2dx + \int_{\mathbb{R}^2}\!\int_{\mathbb{R}^2}\!ln|\omega||G_\varphi f(\omega,b)|^2d\omega db  \geq \|\varphi\|^2_{L^{2}(\mathbb{R}^2,\mathbb{H})}\!\left(\frac{\Gamma^{'}(t)}{\Gamma(t)}-ln\pi\right) \!\int_{\mathbb{R}^2}\!\int_{\mathbb{R}^2}|f(x)|^2dx,
\end{equation}

\end{proof}

\end{document}